\documentclass[12pt]{amsart}

\usepackage{amssymb}
\usepackage{graphicx}

\renewcommand{\leq}{\leqslant}
\renewcommand{\geq}{\geqslant}
\renewcommand{\phi}{\varphi}
\renewcommand{\epsilon}{\varepsilon}
\newcommand{\Vol}{\operatorname{vol}}

\newtheorem{thm}{Theorem}
\newtheorem{lem}[thm]{Lemma}
\newtheorem{prop}[thm]{Proposition}

\newtheorem{definition}[thm]{Definition}

\theoremstyle{remark}
\newtheorem{rem}[thm]{Remark}
\theoremstyle{conjecture}
\newtheorem{conj}[thm]{Conjecture}

\title[Congestion and fair-division problems]{Congestion in networks and manifolds,\\and fair-division problems}
\author{Dong Zhang}
\address{Department of Mathematics, University of Southern California, Los Angeles CA 90089-2532, U.S.A.}
\email{dongzusc@gmail.com}
\thanks{This work was partially supported by the grant DMS-1711297  from the US National Science Foundation.}

\begin{document}
\maketitle
\begin{abstract}
Several large scale networks, such as the backbone of the Internet, have been observed to behave like convex Riemannian manifolds of negative curvature. In particular, this paradigm explains the observed existence, for  networks of this type, of a ``congestion core'' through which a surprising large fraction of the traffic transits, while this phenomenon cannot be detected by purely local criteria. In this practical situation, it is important to estimate and predict the size and location of this congestion core. In this article we reverse the point of view and, motivated by the physical problem, we study congestion phenomena in the purely theoretical framework of convex Riemannian manifolds of negative curvature. In particular, we introduce a novel method of fair-division algorithm to estimate the size and impact of the congestion core in this context. 

\end{abstract}

\section{Introduction}
\label{cha:introduction}
	
\subsection{Congestion on networks}
	Traffic congestion problems are critical in the study of network transportation, from rush-hour traffic jams on city highways to routing data between internet users. With applications to internet traffic, biological and social sciences, and material transportation, understanding the key structural properties of large-scale data networks is crucial for analyzing and optimizing performances, and for improving security and reliability \cite{narayan}.
	
	In recent years, a great amount of empirical results have shown that many different types of data networks share features with negatively curved graphs with small hyperbolicity constants \cite{adcock2013tree, chen2013hyperbolicity, de2011treewidth, jonckheere2008scaled, fb2009, kennedy2013hyperbolicity, lohsoonthorn1969hyperbolic, narayan, shavitt2008hyperbolic}. A consequence of this, consistent with experimental data, is that a large percentage of the traffic between vertices (nodes) tends to go through a relatively small subset of the network. This approach is based on a common and broadly applied method using insights from Riemannian geometry to study large scale networks. In particular, E. Jonckheere, M. Lou, F. Bonahon and Y. Baryshnikov \cite{fb2009} used this paradigm to predict the existence of a congestion core in negatively curved networks, which turned out to be consistent with observational data in \cite{narayan}. 
	
	On a more theoretical level, V. Chepoi, F. Dragan and Y. Vax{\`e}s \cite{chepoi} proved a more quantitative result:  A Gromov $\delta$-hyperbolic space admits a congestion core which intercects at least one-half of all geodesics of the space. They also found that such a core admits a radius of $4\delta$. 
	
	Our goal is to better relate the congestion in a network to its geometric characteristics, such as its scale and curvature. In particular, we want to improve the quantitive measure of the density of congestion, namely, the percentage of all geodesic passing through a core, as well as providing methods to identify the location of the congestion core. 
	
	With is goal in mind, we reverse the point of view and, motivated by the congestion network problems, we consider similar properties for Riemannian manifolds. In particular, we exploit a completely new idea for this type of problem, borrowed from the general area of fair division algorithms; see the Fair-Cut Theorem~\ref{theoremoffaircut} below.

\subsection{The main theorem and its supporting properties}

Our more precise result currently requires that we consider manifolds of constant negative curvature. We believe that a similar property should hold for variable negative curvature. 

Recall that a Riemannian manifold is \emph{convex} if any two points are joined by a unique geodesic arc, whose interior is disjoint from the boundary of the manifold. In particular, a compact convex manifold is always diffeomorphic to a closed ball. 

\begin{thm}[Main Theorem]
\label{theoremofdensity}
Let $M$ be a compact convex $m$-dimensional Riemannian manifold with constant negative  sectional curvature $-k^2$, with $k>0$. Then, there exists a point $x_0\in M$ and a universal radius $r_0=\frac{1}{k}\log\left(\sqrt{2}+1\right)$, such that at least $\frac{1}{m+1}$ of all the geodesics of the manifold pass through the ball $B(x_0,r_0)$.
\end{thm}

The dependence of the estimate on the dimension $m$ is certainly a flaw for applications to network. However, see Conjecture~\ref{conjectureoneovere} in \S \ref{sect:heuristics} for a conjectured uniform bound coming from our approach. 

Our proof of Theorem~\ref{theoremofdensity} relies on two intermediate steps. The first one uses the following fundamental property of spaces of negative curvature. 

In a convex Riemannian manifold $M$ of negative curvature $M$, a point $x\in M$ and a unit vector $v \in T_x^1 M$ determine a \emph{half-space} $H(x,v)$, consisting of all $y\in M$ such that the geodesic $[x,y]$ makes an angle $\leq \frac\pi2$ with $v$ at $x$. 
	
\begin{thm}[Blocked View Theorem]
\label{theoremofblockradius}
Let $M$ be a convex Riemannian manifold of negative sectional curvature bounded above by $-k^2$, with $k>0$. Then, there exist a universal radius $r_0=\frac{1}{k}\log\left(\sqrt{2}+1\right)$ satisfying the following property: for every $x$, $p \in M$, the set of $q \in M$ such that the geodesic $[p,q]$ meets the ball $B(x,r_0)$ contains the whole half-space $H(x,v_p)$, where $v_p$ is the unit tangent vector of the geodesic $[p,x]$ at $x$.
\end{thm}

In other words, the view of $H(x,v_p)$ from the point $p$ is completely blocked by the ball $B(x,r_0)$.  Such a property clearly fails if the curvature is allowed to approach 0. 
	
	This leads us to investigate the volumes of half-spaces $H(x,v)$ in $M$. In Section \ref{cha:thefaircutofapie}, we introduce a geometric quantity that we call the \emph{fair-cut index} of the manifold $M$. It is defined as
\begin{equation*}
	\Phi(M)=\max_{x\in M}  \min_{v\in T_x^1M} \frac{\Vol{H(x,v)}}{\Vol M},
\end{equation*}

The next big idea in the proof of Theorem~\ref{theoremofdensity} is the following. 

\begin{thm}[Fair-Cut Theorem]
\label{theoremoffaircut}
	Let $M$ be a compact convex $m$-dimensional Riemannian manifold with non-positive  constant sectional curvature that is also compact and convex. Then,
\begin{equation*}
		\Phi(M)\geq\frac{1}{m+1}.
\end{equation*}
\end{thm}

Although our proof currently requires the curvature to be constant, this hypothesis is likely to be unnecessary. 
	
	A point $x_0$ where the maximum $\Phi(M)$ is attained is a \textit{fair-cut center} for $M$. The Fair-Cut Theorem can be rephrased as saying that every hyperplane passing through a fair-cut center cuts $M$ into pieces whose volume is at least $\frac{1}{m+1}$ times the volume of the whole manifold.
	
	In Proposition~\ref{prop:FairCutCentersConvex}, we show that the fair-cut centers form a convex subset of $M$. In practice, this set is very often reduced to a single point, and the fair-cut center is unique.

\section{Counting geodesics}
\subsection{Counting geodesics on a graph}

To motivate the Riemannian setup, we first consider the case of graphs (or networks), which provides the motivation for this work. 

A  \emph{graph} $G$ consists of a set $V$ of \emph{vertices} (or \emph{nodes}) and a set $E$ of \emph{edges} (or \emph{links}), such that every edge connects two vertices. The graph is \emph{connected} if any two vertices $p$, $q\in V$ can be connected by a  path in $G$, namely by a finite sequence of edges such that any two consecutive edges share an endpoint.

If we assign a positive length to each edge of $E$ (for instance a uniform length $1$), this defines a \emph{length} for each path in $G$, namely the sum of the lengths of the edges in the path. This defines the metric on the vertex set $V$, where the distance  between  two vertices is the shortest length of a path joining them. A path in $G$ is \emph{geodesic} if its length is  shortest among all paths connecting its endpoints. We are interested in the set $\Gamma(G)$ of (oriented) geodesics of $G$. 


An important case is when the geodesic connecting two vertices $p$, $q\in V$ is unique. In this case, we can label this geodesic as $[p,q]$. Then, the set $\Gamma(G)$ is then the same as the square $V\times V$ of the vertex set $V$. In particular,  
\begin{equation*}
\left|\Gamma(G)\right|=|V|^2,
\end{equation*}
where $|\cdot|$ measures the size of a set.

To study  congestion phenomena in a connected graph $G$, we want to count the number of geodesics on a graph $G=(V,E)$ that pass through a given vertex $x\in V$, or more generally near that vertex, and compare it to the total number of geodesics. 

With this in mind, we introduce  the set
\begin{equation*}
C(x)= \{ \gamma \in \Gamma(G); x\in\gamma\}
\end{equation*}
of \emph{geodesic traffic} passing through the vertex $x \in V$, as well as, for $r>0$, the \emph{geodesic traffic set through the ball $B(x,r)$} 
	\begin{equation*}
	C(x,r)=\{\gamma \in \Gamma(G) ; d(x,\gamma) \leq r \}.
	\end{equation*}
Here	$d(x,\gamma) $ denotes the shortest distance between $x$ and a vertex of the path $\gamma$. 

These are quantified by the numbers
\begin{align*}
D(x)&=\frac{|C(x)|}{|\Gamma(G)|}
&
D(x,r)&=\frac{|C(x,r)|}{|\Gamma(G)|}
\end{align*}
which measure the density of traffic passing through the vertex $x$, or through the ball $B(x,r)$. 

In this discrete setting, the sets $C(x)$ and $C(x,r)$ of course coincide when $r$ is less than the length of the shortest edge of $G$. 

\subsection{Counting geodesics in a Riemannian manifold}
\label{sect:CountingGeodesicsRiemannian}
We want to extend these ideas from graphs and networks to  Riemannian manifolds.

Let $M$ be a compact Riemannian manifold with boundary. It is \emph{convex} if any two points $p$, $q\in M$ can be connected by a unique geodesic $[p,q]$, meeting the boundary only at its endpoints (if at all). In particular, $M$ is then diffeomorphic to a closed ball. 

In this case, we can identify the set  $\Gamma(M)$ of geodesics of $M$ to the product $M \times M$, and it makes sense to quantify the size of $\Gamma(M)$ as  $|\Gamma(M)|=(\Vol{M})^2$ where $\Vol M$ is the volume of $M$. 

By analogy with the case of graphs, we then introduce the \emph{geodesic traffic set through the ball $B(x,r)$} as
\begin{equation*}
C(x,r)=\{ (p,q) \in M \times M; [p,q]  \cap B(x,r)\neq \emptyset \},
\end{equation*}
and 
the \emph{density of the geodesic traffic passing through $B(x,r)$} 
\begin{equation*}
D(x,r)=\frac{\Vol{C(x,r)}}{(\Vol{M})^2}
\end{equation*}
where, for the volume form $d\mu$ of $M$,  the volume $\Vol{C(x,r)}$ is defined by
\begin{equation}
\label{deadintegral} 
\Vol{C(x,r)}=\int_{C(x,r)} d\mu(p)\, d\mu(q).
\end{equation}

Note that, in this manifold setting, we are not interested in the geodesic traffic set passing through a single point, since it has measure 0 for the the volume form of $M\times M$  (except in the trivial case where $\dim(M)\leq1$).

\section{The Blocked View Theorem}
\label{cha:theblockingradius}

We will restrict our attention to Riemannian manifolds of negative sectional curvature, which is the main framework where convexity occurs and is stable under perturbation. The Blocked View Theorem below is typical of negative curvature, and will provide a key estimate for our analysis. 

We begin with a definition. Recall that $[p,q]$ denotes the geodesic arc going from $p$ to $q$. 

\begin{definition}[The blocked view]
\label{blockedview}
Let $M$ be a compact convex  Riemannian manifold. For $p$, $x\in M$ and a radius  $r>0$ the \emph{blocked view set}  $C_p(x,r)$ is
	\begin{equation*}
	C_p(x,r)=\big\{q\in M; [p,q]\cap B(x,r)\neq\emptyset\big\}.
	\end{equation*}
\end{definition}

In other words, $C_p(x,r)$ is the set of points $q$ whose view from $p$ is blocked by the ball $B(x,r)$. 

Then, Equation (\ref{deadintegral}) can be rewritten as
\begin{equation}
\label{eqblockview}
\Vol{C(x,r)}  =\int_{p\in M}\Vol{C_p(x,r)}~d\mu(p).
\end{equation}

A point $x\in M$ and a unit tangent vector $v\in T_x^1M$ determine a \emph{half-space}
\begin{equation*}
H(x,v)=\big\{q\in M;~\langle v_q,v\rangle\leq0 \text{ for  the tangent vector } v_q \text{ of }[q,x]\text{ at }x\big\}.
\end{equation*}

\begin{thm}[Blocked View Theorem]
	\label{blockedviewtheorem}
Let $M$ be a compact convex Riemannian manifold whose sectional curvature is bounded above by $-k^2$ with $k>0$. Then, there exists a universal radius $r_0 =\frac{1}{k}\log(\sqrt{2}+1)$ such that, for any two distinct points $x$, $p\in M$, the blocked view set $C_p(x,r_0)$  contains the  half-space $H(x,v_p)$ determined by the tangent vector $v+p$ of the geodesic $[p,x]$ at $x$.
\end{thm}

In other words, the view of the whole half-space $H(x,v_p)$ from the point $p\in M$ is completely blocked by the ball $B(x,r_0)$. 
We call the universal radius $r_0 =\frac{1}{k}\log(\sqrt{2}+1)$ the \emph{blocking radius} corresponding to the curvature bound $-k^2$.

The proof of Theorem~\ref{blockedviewtheorem}  is based on the following two lemmas. 

\begin{lem}
 \label{theradius}
 In the $m$-dimensional space $H^m_{k^2}$ of constant curvature $-k^2$, consider two geodesics $[\bar x, \bar y]$ and $[\bar x, \bar z]$ making a right angle at $\bar x$. Then, the distance from $\bar x$ to the geodesic $[\bar y,\bar z]$ is uniformly bounded by $\frac1k \log\left( \sqrt2 +1 \right)$. 
 \end{lem}

\begin{proof} After rescaling the metric and restricting attention to a totally geodesic plane containing the three points $\bar x$, $\bar y$ and $\bar z$, we can arrange that $k=1$ and $m=2$, and identify 
 $H^2_{1}$ to the Poincar\'e disk model for the hyperbolic plane. After applying a suitable isometry, we can in addition assume that $\bar x$ coincides with the center $O$ of the disk. Also, moving $\bar y$ and $\bar z$ away from $x$ increases the distance from $\bar x$ to $[\bar y,\bar z]$. It therefore suffices to consider the case where $\bar y$ and $\bar z$ are in the circle at infinity $\partial_\infty H^2_{1}$. In this special case, a simple computation in the Poincar\'e model shows that the distance from $x=O$ to $[\bar y,\bar z]$ is exactly equal to $\log\left( \sqrt2 +1 \right)$. This provides the required bound in the general case. 
\end{proof}

\begin{lem}
\label{comparisonlemma}
	Let $M$ be a convex Riemannian manifold whose sectional curvature is uniformly bounded above by $ -k^2<0$. Given a geodesic triangle $xyz$ in $M$ with a right angle at $x$ consider, in the space $H^m_{k^2}$ of constant curvature $-k^2$, a triangle $ \bar{x}\bar{z}\bar{p}$ with a right angle at $\bar x$ and whose legs are such that $d(x,y)=d(\bar{x},\bar{y})$ and $d(x,z)=d(\bar{x},\bar{z})$. Then, the distance from $x$ to the geodesic $[y,z]$ is less than or equal to the distance from $\bar{x}$ to $[\bar{y},\bar{z}]$. Namely, 
	\begin{equation*}
	d\big(x,[y,z]\big)\leq d \big(\bar{x},[\bar{y}, \bar{z}] \big).
	\end{equation*}
\end{lem}
\begin{figure}[h]
	\centering
	\includegraphics[scale=0.6]{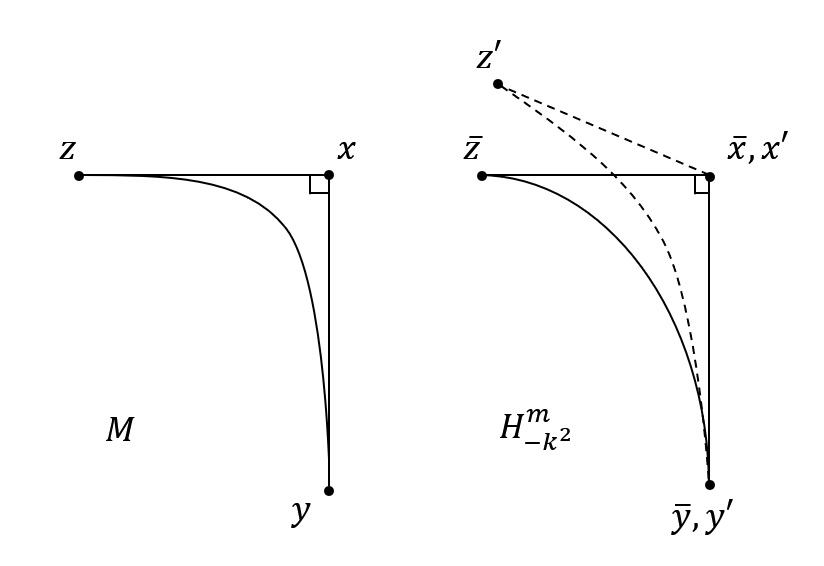}
	\caption{Comparing the triangles $ xzp$, $ x'z'p'$ and $ \bar{x}\bar{z}\bar{p}$}
	\label{comparerighttriangles} 
\end{figure}
\begin{proof} Consider another comparison triangle $ x'y'z'$ in $H^m_{k^2}$, where $d(x,y)=d(x',y')$, $d(x,z)=d(x',z')$ and $d(y,z)=d(y', z')$. Using Proposition 1.7 of \cite[Chap. II.1]{bridson2013metric}, , 
	\begin{equation}
	\label{eqn:comparison1}
	d(x,[y,z])\leq d(x',[y', z']).
	\end{equation}
	
	By Toponogov's theorem (see  for instance \cite[Chap. 2]{cheeger}), the angle of $x'y'z'$ at $x'$ is greater than the angle of $xyz$ at $x$, so it is larger than $\pi/2$. As a consequence, in the constant curvature space $H^m_{k^2}$, the geodesic triangles $x'y'z'$ and $\bar x \bar y \bar z$ are such that $d(x',y')=d(\bar{x},\bar{y})$,  $d(x',z')=d(\bar{x},\bar{z})$, and the angle of $x'y'z'$ at $x'$ is greater than the angle of $ \bar{x}\bar{y}\bar{z}$ at $\bar{x}$. Moving these comparison triangles by isometries of $H^m_{k^2}$, we can arrange that  $x'y'z'$ and $ \bar{x}\bar{y}\bar{z}$ are contained in the same 2-dimensional space $H^2_{k^2}\subset H^m_{k^2}$, modeled as the Poincar\'e disk. In addition, we can arrange that $x'=\bar x$, $y'=\bar y$, and the two triangles are on the same same side of $[x',y']=[\bar x, \bar y]$, as illustrated in Figure \ref{comparerighttriangles}. Since the legs $[x',z']$ and $[\bar{x},\bar{z}]$ are also of the same length, a simple geometric argument in the Poincar\'e disk shows that
\begin{equation}
	\label{eqn:comparison2}
 d(x',[y',z'])\leq d(\bar{x},[\bar{y},\bar{z}]).
\end{equation}
	
The combination of (\ref{eqn:comparison1}) and (\ref{eqn:comparison2}) completes the proof. 
\end{proof}

We are now ready to prove Theorem~\ref{blockedviewtheorem}. 

\begin{proof} [Proof of the Blocked View Theorem~\ref{blockedviewtheorem}]
Remember that we are trying to show that, for  $r_0 =\frac{1}{k}\log(\sqrt{2}+1)$, the blocked view set 
$$
C_p(x,r)=\big\{q\in M; [p,q]\cap B(x,r)\neq\emptyset\big\}
$$
contains, for the tangent vector $v_p$ of the geodesic $[p,x]$ at $x$, the half-space
$$
H(x,v_p)=\big\{q\in M;~\langle v_p,v_q\rangle\leq0 \text{ for the tangent vector } v_q \text{ of }[q,x]\text{ at }x\big\}.
$$

With this goal in mind, consider a point $q\in H(x,v_p)$. Since $p$ is in the complement of the half-space $H(x,v_p)$, there exists a point $z$ in the intersection of the geodesic $[p,q]$ and of the boundary $\partial H(x,v_p)$. By construction, the triangle $xpz$ has a right angle at the vertex $x$. 

The combination of Lemmas~\ref{theradius} and \ref{comparisonlemma} then shows that $x$ is at distance at most $r_0$ from the geodesic $[p,z]$, and therefore from the geodesic $[p,q]$. As a consequence, the view from $p$ to $q$ is blocked by the ball $B(x,r_0)$, and $q$ belongs to the blocked view set $C_p(x,r)$. 

This concludes the proof of the Blocked View Theorem~\ref{blockedviewtheorem}. 
\end{proof}

\section{The fair cut of a pie}
\label{cha:thefaircutofapie}

This section is now devoted to an apparently unrelated problem. The connection with the congestion problem will be explained in \S \ref{cha:themainestimate}. 

The issue is a fair-division scheme for a pie. Suppose that Alice and Bob want to split a cake, and that each of them wants to optimize the size of their share. Alice decides a point through which the knife should cut, and Bob decides in which direction to apply the cut. Alice knows that, wherever she picks a point, Bob will choose the cut through this point that will maximize his share, and consequently minimize Alice's share. So Alice's goal is to find a point where any cut will guarantee her an optimum share of the cake. We call such a point a ``fair-cut center''. 

In our case, the cake is replaced by a convex Riemannian manifold $M$ of negative curvature, and the knife cut at the point $x\in M$ occurs along a geodesic hyperplane $\partial H(x,v)$. 

\subsection{Definitions and the Fair-Cut Theorem}

Let $M$ be a compact convex Riemannian manifold.
For a point $x\in M$ and a unit vector $v\in T_x^1M$, the {half-space}
$$
H(x,v)=\big\{q\in M; \langle v_q,v\rangle\leq0 \text{ for  the tangent vector } v_q \text{ of }[q,x]\text{ at }x\big\}
$$
is bounded by the \emph{geodesic hyperplane}
$$
\partial H(x,v)=\big\{q\in M; \langle v_q,v\rangle= 0 \text{ for  the tangent vector } v_q \text{ of }[q,x]\text{ at }x\big\}
$$

\begin{definition}[A fair cut of the pie]	
	Let $M$ be a compact, convex $m$-dimensional Riemannian manifold with non-positive sectional curvature. The \emph{fair-cut index} of $M$ is the number
	\begin{equation}
	\label{faircutindexsimple}
	\Phi(M)=\max_{x\in M}{\min_{v\in T_x^1M}{\frac{\Vol{H(x,v)}}{\Vol{M}}}}.
	\end{equation}
\end{definition}

In other words, if we consider the function $f_x\colon T_x^1M\longrightarrow[0,1]$ defined by 
$$
f_x(v)=\frac{\Vol{H(x,v)}}{\Vol{M}},
$$
which measures the percentage of the  pie corresponding to $H(x,v)$, and the function $\varphi_M \colon M\longrightarrow[0,1]$  defined as the minimum
$$
\varphi_M(x)=\min_{v\in T_x^1M}{f_x(v)},
$$
then the fair cut index is
\begin{equation}
\label{faircutindex}
\Phi(M)=\max_{x\in M}\varphi_M(x)=\max_{x\in M}{\min_{v\in T_x^1M}{f_x(v)}}
=\max_{x\in M}  \min_{v\in T_x^1M} \frac{\Vol{H(x,v)}}{\Vol M}.
\end{equation}

We will obtain the following estimate.
\begin{thm}[Fair-cut Theorem]
	\label{cutpiethm}
	Let $M$ be a compact, convex $m$-dimensional Riemannian manifold with constant non-positive sectional curvature.
	
	Then,
	\begin{equation*}
	\frac{1}{m+1}\leq\Phi(M)\leq\frac{1}{2}.
	\end{equation*}
\end{thm}

The upper bound is an immediate consequence of the observation that
\begin{equation*}
\Vol{H(x,v)}+\Vol{H(x,-v)}=\Vol{M}.
\end{equation*} 

The lower bound $\frac{1}{m+1}$ will require more elaborate arguments, described in the next sections.

Note that there exists manifolds $M$ for which the upper bound $\Phi(M) = \frac12$ is achieved. This will happen when $M$ is radially symmetric about a point $x_0$, in the sense that for every $p\in M$ there is a point $q\in M$ such that $x_0$ is the midpoint of the geodesic arc $[p, q]$. Indeed, in this case, $\Vol{H(x_0,v)} =\Vol{H(x_0,-v)}= \frac12 \Vol{M}$ for every $v\in T^1_{x_0}M$. 

We begin with a couple of lemmas, in the next two sections

\subsection{Lipschitz continuity for the volume function}
We will need an estimate on the local variation of the volume function $ \Vol{H(\cdot,\cdot)}:T^1M\longrightarrow\mathbb{R}_+ $. 

\begin{lem}[Lipschitz bound for the volume function]
	\label{lipschitzcontinuity}
	Let $M$ be a compact, convex $m$-dimensional Riemannian manifold with non-positive sectional curvature bounded in an interval $[-k_1^2, 0]$ with $k_1\geq 0$. Suppose that $t\mapsto  (x(t),v(t))$ is a smooth curve in the unit tangent bundle of $M$. Then 
	\begin{equation}
	\Big|\frac{d}{dt}\Vol{H(x(t),v(t))}\Big|\leq C(k_1, D) \Big| \frac d{dt} \big(x(t), v(t) \big) \Big|
	\end{equation}
where $C(k_1, D)$ is a constant depending only on the lower curvature bound $-k_1^2$ and on an upper bound $D$ for the diameter of $M$.
\end{lem}

\begin{proof}
	The proof uses the  property that
	\begin{equation*}
	\frac{d}{dt}\Vol{H \big( x(t),v(t) \big)} =\int_{\partial H(x(t),v(t))}\langle N(y(t)),J(y(t))\rangle ~d\mu(t),
	\end{equation*}
	where $d\mu(t)$ is the volume form on the hyperplane, $y(t)$ is any point on the hyperplane, $N(y(t))$ is the normal vector of $\partial H(x(t),v(t))$ at $y(t)$, and $J(y(t))=\frac{d}{dt}y(t)$. A proof can for instance be found in \cite[Eqn. (7.2)]{flanders}.
	
	Also, standard comparison arguments with Jacobi fields gives
	\begin{equation*}
	\big|\langle N(y(t)),J(y(t))\rangle\big|\leq\big|J(y(t))\big|\leq\big|\cosh_{k_1}(D)+\sinh_{k_1}(D)\big|,
	\end{equation*}
where the functions $\cosh_{k_1}$ and $\sinh_{k_1}$ are defined as
\begin{align*}
 \cosh_{k_1}(x)& =   \cosh(k_1 x) &
 \sinh_{k_1}(x) &=  \frac1{k_1} \sinh(k_1 x).
\end{align*}
Finally, the volume $\Vol{\partial H(x(t),v(t))}$ is bounded by a universal constant times $\sinh_{k_1}(D)^{n-1}$. The required property then follows from these estimates. 
\end{proof}

In particular, the function $\varphi_M$ is continuous. By compactness of the unit tangent bundle $T^1M$, it attains its maximum at a point $x_0 \in M$ such that
\begin{align*}
\varphi_M(x_0) &= \Phi(M)=\max_{x\in M}{\varphi_M(x)} \\
& = \max_{x\in M} \min_{v\in T_{x_0}^1M}{f_{x_0}(v)} 
= \max_{x\in M} \min_{v\in T_{x_0}^1M}{\frac{\Vol{H(x_0,v)}}{\Vol{M}}}.
\end{align*}

\begin{definition}
	Let $M$ be a compact, convex $m$-dimensional Riemannian manifold with non-positive sectional curvature. A point $x_0\in M$ such that $\phi_M(x_0)=\Phi(M)$ is a  \emph{fair-cut center} for $M$.
\end{definition}

\subsection{Moving half-spaces to their interiors}

\begin{lem}
	\label{inneranglelemma}
	Let $M$ be an $m$-dimensional Riemannian manifold of negative sectional curvature, then the sum of the angles of a geodesic triangle in $M$ is less than $\pi$.
\end{lem}
\begin{proof}
	This is an easy application of comparison theorems; see for instance \cite[Prop II.1.7] {bridson2013metric}. 
\end{proof}

The next lemma requires the curvature of our manifold $M$ to be constant. 

\begin{lem}
	\label{insidethehalfspace}
	In an $m$-dimensional Riemannian manifold $M$ of constant negative sectional curvature, let $H(x_0,v_0)\subset M$ be the half-space defined by a point $x_0\in M$ and a unit vector $v_0\in T_{x_0}^1M$. For any $x$ in the interior of the half-space $H(x_0,v_0)$, there exists a vector  $v\in T_{x}^1M$ such that $H(x,v)$ is contained in the interior of $ H(x_0,v_0)$. In particular, $\Vol H(x,v)< \Vol H(x_0,v_0)$.
	
	Similarly, if $x$ is in $M\setminus H(x_0,v_0)$, there exists  $v\in T_{x}^1M$ such that the interior of $H(x,v)$ contains $ H(x_0,v_0)$, and  $\Vol H(x,v)> \Vol H(x_0,v_0)$.
\end{lem}
\begin{proof} Let $z $ be a point in the boundary $ \partial H(x_0,v_0)$ that is closest to $x$, and let $w\in T_z^1M$ be the vector tangent to the geodesic $[z,x]$ at $z$. Because the curvature is constant, $H(z,w) = H(x_0, v_0)$. If $v\in T_x^1M$ is tangent to $[z,x]$ at $x$, we conclude that $H(x,v)$  is contained in the interior of $H(z,w) = H(x_0, v_0)$. (Otherwise, one would see a triangle with two right angles, which is excluded by the negative curvature.)
	
The second part of the statement is proved in a similar way. 	
\end{proof}

\begin{rem}
 This is the only point where we need the curvature of $M$  to be constant. It is quite likely that a similar statement can be proved under a weaker hypothesis, for instance the classical curvature pinching property that the sectional curvature is in an interval $[-4a^2, -a^2]$ with $a>0$. 
\end{rem}

\subsection{Minimizing directions at a fair-cut center}
We now focus on a fair-cut center $x_0$ for the manifold $M$.

\begin{prop}
\label{propcovering}
Suppose that $M$ is a compact convex manifold with constant negative curvature, and let $x_0\in M$ be a fair-cut center. Let
\begin{align*}
V_{x_0}&=\big\{v\in T_{x_0}^1M;  f_{x_0}(v)=\phi_M(x_0)=\Phi(M)\}
\\
&=\big\{v\in T_{x_0}^1M;  \Vol{H(x_0,v)}  =\Phi(M) \Vol(M)\}
\end{align*}
be the set of minimizing directions at $x_0$. Then
$$
M=\bigcup_{v\in V_{x_0}}H(x_0,v).
$$	
\end{prop}

\begin{proof}
 Suppose, in search of a contradiction, that the half-spaces $H(x_0,v)$ with $v\in V_{x_0}$ do not cover all of $M$. We will then show that $x_0$ is not a local maximum of the function $\phi_M$, contradicting its definition. 
 
 If there exists a point $p\in M$ that is not in the union of the $H(x_0,v)$ with $v\in V_{x_0}$, the tangent of the geodesic arc $[x_0, p]$ at $x_0$ provides a unit tangent vector $u_0 \in T_{x_0}^1M$ such that $\langle u_0 , v \rangle<0$ for every $v\in V_{x_0}$.
 
 In particular, the set $\{ v \in T_{x_0}^1M; \langle u_0 , v \rangle<0 \}$ is an open neighborhood of the minimizing set $V_{x_0} =\{ v \in T_{x_0}^1M;  \Vol{H(x_0,v)}  = \Phi(M) \Vol(M) \}$. By compactness of $V_{x_0}$, there consequently exists an $\alpha_0>0$ such that 
 $$
\{ v \in T_{x_0}^1M; \Vol{H(x_0,v)}   < \Phi(M)  \Vol(M) + \alpha_0 \} \subset \{ v \in T_{x_0}^1M; \langle u_0 , v \rangle<0 \}
 $$
 
 In other words, for every $v_0 \in T_{x_0}^1M$, either 
\begin{equation}
\label{eqn:LargeVolume}
  \Vol{H(x_0,v_0)} \geq \Phi(M) \Vol M  + \alpha_0
\end{equation}
or 
\begin{equation}
\label{eqn:NegativeDotProduct}
 \langle u_0 , v_0 \rangle<0.
\end{equation}

Let $g \colon (-\epsilon, \epsilon) \to M$ be a small geodesic arc with $g(0)=x_0$ and $g'(0)=u_0$. For $t>0$, set $x_t = g(t)$. 

By definition of the function $\phi_M(x)= \inf_{v\in T_x^1M} (\Vol H(x,v))/(\Vol M)$ and since $x_0$ realizes the maximum of this function, there exists $v_t \in T_{x_t}^1M$ such that
$$
\Vol H(x_t, v_t) = \phi_M(x_t) \Vol M \leq  \phi_M(x_0) \Vol M = \Phi(M) \Vol M. 
$$
Let $v_0 \in T_{x_0}^1M$ be obtained by parallel translating $v_t$ along the geodesic $g$. 

The  Lipschitz continuity property of Lemma~\ref{lipschitzcontinuity} shows that, provided we chose $x_t$ sufficiently close to $x_0$ (depending only on the constant $\alpha_0>0$ arising in  (\ref{eqn:LargeVolume}) ), we have that 
\begin{align*}
\Vol H(x_0, v_0) &\leq \Vol H(x_t, v_t) + {\textstyle\frac12} \alpha_0
\\
&\leq   \Phi(M) \Vol M + {\textstyle\frac12} \alpha_0 
\end{align*}
by choice of $v_t$. As a consequence,  (\ref{eqn:LargeVolume}) cannot hold. 

Therefore, $v_0$ satisfies (\ref{eqn:NegativeDotProduct}). Since $v_t$ is obtained by parallel translating $v_0$ along the geodesic $g$, $\langle v_t, g'(t) \rangle = \langle v_0, u_0 \rangle <0$.  Lemma~\ref{insidethehalfspace}
then provides another vector $w_0 \in T_{x_0}^1M$ such that 
$$
\Vol H(x_0, w_0) < \Vol H(x_t, v_t) \leq  \phi_M(x_0) \Vol M . 
$$
However, the existence of $w_0$ would contradict the fact that $\phi_M(x_0)$ is defined as the infimum of $\Vol H(x_0, v) /  \Vol M$ over all $v\in T_{x_0}^1M$. 

This final contradiction concludes the proof of Proposition~\ref{propcovering}. 
 \end{proof}

 We now improve Proposition~\ref{propcovering}, by bounding the number of half-spaces $H(x_0,v)$, with $v\in V_{x_0}$, needed to cover the manifold $M$. 
 
 This is based on the following elementary observation. 
 
\begin{lem}
\label{UnionHalfspacesAndConvexHull}
In a compact convex manifold $M$ with  nonpositive curvature, let $V$ be a subset of  the unit tangent space $T_{x_0}^1M$. The following properties are equivalent:
\begin{enumerate}
 \item $ M$ is the union of the half-spaces $H(x_0, v)$ as $v$ ranges over all vectors of~$V$;
 \item for every $w\in T_{x_0}^1M$, there exists $v\in V$  with $\langle v,w \rangle \geq 0$;
 \item   in the vector space $T_xM$, the point $0$ is  in the convex hull $\mathrm{Conv}(V)$ of $V$. 
\end{enumerate}
\end{lem}

\begin{proof}
The equivalence of (1) and (2) is easily seen by considering, for every $x\in M$, the tangent vector $w$ of the geodesic $[x_0, x]$ at $x_0$. 

The equivalence of (2) and (3) is an elementary property of convex sets in $\mathbb{R}^m$. 
\end{proof}
 
\begin{prop}
\label{propcoveringwithbound}
 Let $M$ be a compact convex manifold with constant negative curvature, and let $x_0\in M$ be its fair-cut center. Then, there exists $n$ vectors $v_1$, $v_2$, \dots, $v_n \in V_{x_0}$ in the mininimizing set $V_{x_0} \subset T_{x_0}^1 M$, with $n\leq \dim M +1$, such that 
 $$
 M= \bigcup_{i=1}^n H(x_0, v_i). 
 $$
\end{prop}
\begin{proof} By  Proposition~\ref{propcovering},
$$
M=\bigcup_{v\in V_{x_0}}H(x_0,v).
$$
Lemma~\ref{UnionHalfspacesAndConvexHull} then shows that $0$ is in the convex hull of $V_{x_0}$. By Caratheodory's theorem \cite{caratheo}, there exists a subset $\{v_1, v_2, \dots, v_n\} \subset V_{x_0}$ of cardinal $n\leq \dim M +1$ whose convex hull also contains $0$. Another application of Lemma~\ref{UnionHalfspacesAndConvexHull} then shows that $M$ is the union of the $H(x_0, v_i)$ with $i=1$, $2$, \dots, $n$. 
\end{proof}

\subsection{Proof of the Fair-Cut Theorem}

We are now ready to prove the Fair-Cut Theorem \ref{cutpiethm}. We already observed that $\Phi(M) \leq \frac12$, so we just need to restrict attention to the lower bound. . 

By definition of the fair-cut center $x_0\in M$ and of the minimizing set $V_{x_0} \subset T_{x_0}^1M$, 
$$
\frac{\Vol{H(x_0,v)}}{\Vol{M}} = \phi_M(x_0) = \Phi(M).
$$
for every $v\in V_{x_0}$. 

Proposition~\ref{propcoveringwithbound} then shows that there exists $\{v_1, v_2, \dots, v_n\} \subset V_{x_0}$ with $n\leq m+1$ such that
$$
M=\bigcup_{i=1}^nH(x_0,v_i).
$$
As a consequence,
$$
\Vol{M}\leq \sum_{i=1}^{n}\Vol{H(x_0,v_i)} = n \Phi(M) \Vol{M} \leq (m+1)  \Phi(M) \Vol{M}. 
$$

This proves that $\Phi(M) \geq \frac1{m+1}$. 
\qed

\begin{rem}
 The only place where we used the condition that the sectional curvature is constant was in the proof of Lemma~\ref{insidethehalfspace}. It seems quite likely that Proposition \ref{propcoveringwithbound} and Theorem \ref{cutpiethm} hold  without this hypothesis. 
\end{rem}

\section{Proof of the Main Theorem}
\label{cha:themainestimate}

We now combine the Blocked View Theorem~\ref{blockedviewtheorem} and the Fair-Cut Theorem~\ref{cutpiethm} to provide an estimate on the percentage density of geodesic traffic.

\begin{thm}
\label{themaintheoremofdensity}
	Let $M$ be an $m$-dimensional compact convex Riemannian manifold of constant negative  sectional curvature $-k^2$, with $k>0$. Then, there exists  a universal radius $r_0=\frac{1}{k}\log(\sqrt{2}+1)$ and a point $x_0\in M$ such that at least $\frac{1}{m+1}$ of all the geodesics of the manifold pass through the ball $B(x_0,r_0)$.
\end{thm}
\begin{proof}
Let $x_0$ be the fair-cut center of $M$ provided by the Fair-Cut Theorem~\ref{cutpiethm}, and let $r_0=\frac{1}{k}\log(\sqrt{2}+1)$ be the radius of the Blocked View Theorem~\ref{blockedviewtheorem}. As in Section~\ref{sect:CountingGeodesicsRiemannian}, let
\begin{equation*}
C(x_0,r_0)=\{(p,q)\in M\times M; [p,q] \cap B(x_0,r_0)\neq \emptyset\}
\end{equation*}
be the set of geodesics $[p,q]$ of $M$ passing through the ball $B(x_0,r_0)$, and for $p\in M$ let
\begin{equation*}
C_p(x_0,r_0)=\{ q\in M ; [p,q] \cap B(x_0,r_0)\neq \emptyset\}
\end{equation*}
be the set of points whose view from $p$ is obstructed by $B(x_0, r_0)$. 

We saw in  Equation \ref{eqblockview} that
\begin{equation*}
\Vol{C(x_0,r_0)}=\int_{p\in M}\Vol{C_p(x_0,r_0)}~d\mu(p).
\end{equation*}
where $d\mu$ is the volume form of $M$.

For a given $p\in M$, the Blocked View Theorem~\ref{blockedviewtheorem} asserts that $C_p(x_0, r_0)$ contains a half-space $H(x_0, v_p)$, so that
$$
\Vol{C_p(x_0,r_0)} \geq \Vol H(x_0, v_p).
$$

By definition of the fair-cut center $x_0 \in M$ and by the Fair-Cut Theorem~\ref{cutpiethm},
$$
\Vol H(x_0, v_p) \geq \phi_M(x_0) \Vol M  \geq \Phi(M)\Vol M \geq \frac1{m+1} \Vol M.
$$

Combining these inequalities then gives 
\begin{equation*}
\Vol{C(x_0,r_0)}=\int_{p\in M}\Vol{C_p(x_0,r_0)}~d\mu(p) \geq  \frac1{m+1} (\Vol M)^2.
\end{equation*}

As a consequence, at least $\frac{1}{m+1}$ of all the geodesics of the manifold pass through the ball $B(x_0,r_0)$.
\end{proof}

\begin{definition}
	The ball  $B(x_0,r_0)$ is the \emph{congestion core} of $M$.
\end{definition}

	Note that the size of this congestion core is uniquely determined by the curvature, the dimension of the manifold provides an estimate of the density of the congestion, while the global geometry of the manifold contributes to the location of the core.

\section{Additional comments}

We conclude with a few observations and conjectures. 

\subsection{The set of fair-cut centers} 

\begin{prop}
\label{prop:FairCutCentersConvex}
 In a compact convex manifold of constant nonpositive curvature, the set of fair-cut centers is convex. 
\end{prop}

\begin{proof}
Let $x_1$ and $x_2$ be two fair-cut centers for $M$. We want to show that every point $x$ in the geodesic arc $[x_1, x_2]$ is also a fair-cut center. 
 
By Proposition~\ref{propcovering}, $x_2$ belongs to some minimizing half-space  $H(x_1, v)$, namely a half-space such that
$$
\frac{\Vol H(x_1, v)}{\Vol M} = \phi(x_1) = \Phi(M). 
$$

We claim that, for any such minimizing half-space  $H(x_1, v)$ containing $x_2$, the point $x_2$ is necessarily on the boundary $\partial H(x_1, v)$. Indeed, if $x_2 $ was not in $ H(x_1, v)$, let $p$ be the point of $\partial H(x_1, v)$ that is closest to $x_2$ and let $w \in T_{x_2}^1M$ be the unit tangent vector to the geodesic arc $[x_2, p]$. Because the curvature is constant, the half-space $H(x_2, -w)$ is strictly contained in $H(x_1, v)$. In particular, $\Vol H(x_2, -w) < \Vol H(x_1, v)$, this would imply that 
$$
\phi(x_2) < \phi(x_1) = \Phi(M) = \phi(x_2),
$$
a contradiction. 

Let $x \in [x_1, x_2]$ be different from $x_1$ and $x_2$, and let $H(x,v)$ be a minimizing hyperplane for $x$. The same argument as before shows that $x_2$ cannot be contained in the interior of $H(x,v)$, as this would again provide the contradiction
$$
\phi(x_2) < \phi(x) \leq \Phi(M) = \phi(x_2). 
$$
Therefore, $x_2$ is in the boundary of $H(x,v)$ and, since the curvature is constant, $H(x,v) = H(x_2, v_2)$ for some $v_2\in T_{x_2}^1M$. Then,
$$
\phi(x) = \frac{\Vol H(x,v)}{\Vol M} = \frac{\Vol H(x_2,v_2)}{\Vol M} \geq \phi(x_2)= \Phi(M),
$$
from which we conclude that $x$ is also a maximum of the function $\phi$, namely is also a fair-cut center. 
\end{proof}

In fact, we conjecture the much stronger result that the fair-cut center is unique.

\subsection{Heuristics about the fair-cut index $\Phi(M)$}
\label{sect:heuristics}

Our lower bound $\frac 1{m+1}$ for the fair-cut index $\Phi(M)$ seems far from being sharp. A heuristic argument suggests a lower bound that is independent of the dimension, which would also improve our congestion estimates. We briefly discuss this argument. 

In a given dimension $m$, we can try to find a manifold $M$ that approximates the infimum of $\Phi(M)$ over all $m$--dimensional convex manifolds of negative curvature. Because $\Phi(M)$ is invariant under rescaling of the metric by a positive scalar, it makes sense to assume that such an approximatively minimizing manifold exists in curvature 0. 
Then, by trial and error based on the Marching Hyperplanes method of the next section, it seems that the infimum  in this curvature 0 case is realized by a simplex $\Delta_n$ in Euclidean space  $\mathbb R^m$. 
Since any two simplices in $\mathbb R^m$ are equivalent under an affine isomorphism, they have the same fair-cut index $\Phi(\Delta_n)$. 

The set of fair-cut centers is invariant under all the symmetries of the simplex, and is convex by Proposition~\ref{prop:FairCutCentersConvex}. It follows that the barycenter $x_0$ of $\Delta_n$ is necessarily a fair-cut center.

\begin{conj}
Let $\Delta_n$ be a simplex in the Euclidean space $\mathbb R^m$, with nonempty interior. Then
 $$\Phi(\Delta_n)=\left (\frac{m}{m+1} \right)^m.$$
\end{conj}

This is equivalent to the statement that, for the barycenter $x_0$ of $\Delta_n$, the minimizing set $V_{x_0} \subset T_{x_0}^1 \Delta_n$ consists of all unit vectors pointing towards the vertices of $\Delta_n$.

Note that $ \left (\frac{m}{m+1} \right)^m$ is a decreasing function of $m$, and converges to $\mathrm e$ as $m$ tends to $\infty$. All these considerations lead us to the following conjecture.

\begin{conj}
\label{conjectureoneovere}
The fair-cut index $\Phi(M)$ of any compact convex manifold $M$ with non-positive sectional curvature satisfies the sharp inequality
$$
	\Phi(M)\geq\frac{1}{\mathrm e}.
$$
\end{conj}

\subsection{A method to estimate the fair-cut centers}
\label{section:marchinghypplanes}

The existence of fair-cut centers was abstractly established by minimizing the function $\phi(x)$. In practice, it may be useful to have a rough estimate of the location of these fair-cut centers. For this, we can use the following consequence of our Main Theorem \ref{theoremofdensity}.

\begin{lem}
\label{marchedover}
If $\Vol{H(x_1,v_1)}<\frac{1}{m+1}\Vol{M}$ for some $v_1\in T_{x_1}^1M$, then any fair-cut center $x_0$ is located outside of the half-space $H(x_1,v_1)$. 
\end{lem}
\begin{proof} Suppose not, meaning that $x_0$ is located inside $H(x_1,v_1)$. Then, let $x_2$ be the projection of $x_0$ to $\partial H(x_1,v_1)$, and let the vectors $v_2\in T_{x_2}^1 M$ and $v_0\in T_{x_0}^1 M$ be tangent to the geodesic arc $[x_2, x_1]$. Then, $H(x_0, v_0)$ is contained in $H(x_2, v_2) = H(x_1, v_1)$, and
$$
\phi_M(x_0) \leq \frac{\Vol H(x_0, v_0)}{\Vol M} \leq  \frac{\Vol H(x_1, v_1)}{\Vol M}< \frac{1}{m+1} \leq \Phi(M) ,
$$
contradicting the fact that $\phi_M(x_0)  = \Phi(M)$. 
\end{proof}

Now let us provide a procedure which we call the \emph{Marching Hyperplanes Method}, in attempt to locate the whereabouts of the fair-cut center.

\begin{figure}[h]
	\centering
	\includegraphics[scale=0.6]{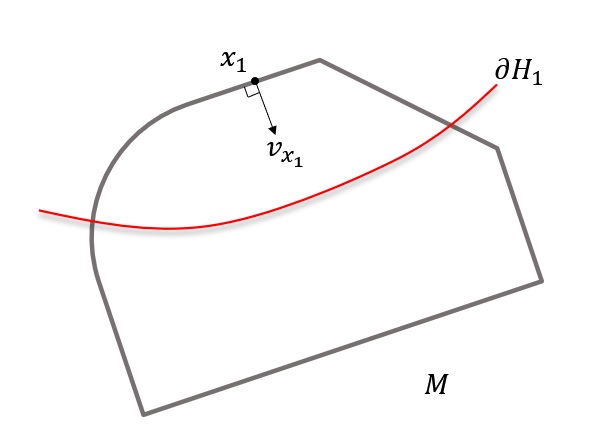}
	\caption{One marching hyperplane}
	\label{onemarchinghyperplane} 
\end{figure}

\textbf{Step 1}: Start from a point $x_1$ on the boundary of $M$, pick a direction $v_{x_1}$ that is perpendicular to $\partial M$ at $x_1$ and point inward, then march forward inside $M$ along the geodesic $g_1$ starting from $x_1$ following $v_{x_1}$, until reached a point $x_{1,0}$ such that the half-space $H(x_{1,0},-v_{x_{1,0}})$ has the volume of $\frac{1}{m+1}\Vol{M}$, where $v_{x_{1,0}}$ is the parallel translation of $v_{x_1}$ along $g_1$. We mark $\partial H_1=\partial H(x_{1,0},-v_{x_{1,0}})$, as shown in Figure \ref{onemarchinghyperplane}.

\textbf{Step 2 to m+1 and maybe more}: Pick points $\{x_i\}$, $i =2,...,m+1$, and maybe more, on $\partial M$, together with directions $v_{x_i}$ that is perpendicular to $\partial M$ at $x_i$ and point inward, then repeat the marching forward as \textbf{Step 1}, so we end up with lots of marked hyperplanes $\{\partial H_i\}$, $i =2,...,m+1$, and maybe more.

\textbf{Final Step}: By Proposition \ref{marchedover}, the fair-cut center $x_0$ is outside any half-spaces that we have marched over, namely, it is  located inside the entity that is bounded by all the hyperplanes, as shown in Figure \ref{manymarchinghyperplanes}.

Although this method does not provide a precise location the fair-cut center, for a small amount of steps, it does give us a very refined vicinity to locate the fair-cut center. In practice, I will suggest using the lower bound in Conjecture \ref{conjectureoneovere}, i.e., $\frac{1}{e}$, instead of $\frac{1}{m+1}$, when the dimension increases.

\begin{figure}[h]
	\centering
	\includegraphics[scale=0.6]{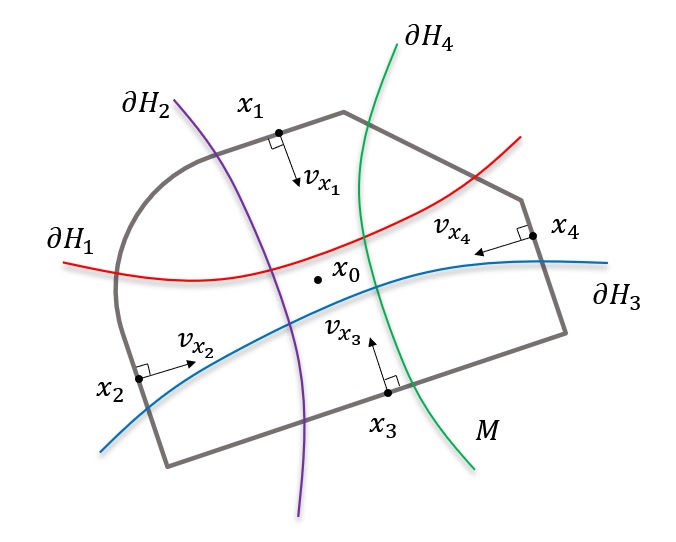}
	\caption{Many marching hyperplanes}
	\label{manymarchinghyperplanes} 
\end{figure}

\bibliography{references}{}
\bibliographystyle{plain}

\end{document}